\documentclass{amsart}
\newcommand{\w}{\omega}
\newcommand{\K}{\mathcal K}
\newcommand{\II}{\mathbb I}
\newcommand{\IR}{\mathbb R}
\newcommand{\Z}{\mathcal Z}
\newcommand{\id}{\mathrm{id}}
\newcommand{\U}{\mathcal U}
\newcommand{\F}{\mathcal F}
\newcommand{\C}{\mathcal C}
\newcommand{\W}{\mathcal W}
\newcommand{\St}{\mathcal St}
\newcommand{\rank}{\mathrm{rank}}
\newcommand{\diam}{\mathrm{diam}}
\newcommand{\V}{\mathcal V}

\newcommand{\ulim}{\textstyle{\bigcup^\infty}}

\newcommand{\Ra}{\Rightarrow}
\newcommand{\pr}{\mathrm{pr}}
\newcommand{\mesh}{\mathrm{mesh}}

\newtheorem{theorem}{Theorem}
\newtheorem{conjecture}{Conjecture}
\newtheorem{corollary}{Corollary}
\newtheorem{proposition}{Proposition}
\newtheorem{problem}{Problem}

\title{Universal meager $F_\sigma$-sets in locally compact manifolds}
\author{Taras Banakh and Du\v san Repov\v s}
\address{T.Banakh: Jan Kochanowski University in Kielce (Poland) and Ivan Franko National University of Lviv (Ukraine)}
\email{t.o.banakh@gmail.com}
\address{D.Repov\v s: Faculty of Mathematics and Physics, and
Faculty of Education, University of Ljubljana (Slovenia)}
\email{dusan.repovs@guest.arnes.si}
\keywords{Universal nowhere dense subset, Sierpi\'nski carpet, Menger cube, Hilbert cube manifold,  $n$-manifold, tame ball, tame decomposition}
\subjclass[2010]{57N20, 57N45, 54F65}
\thanks{This research was supported by the Slovenian Research Agency grants P1-0292-0101 and J1-4144-0101. The first author has been partially financed by NCN means granted by decision DEC-2011/01/B/ST1/01439.}

\begin{document}
\begin{abstract} In each manifold $M$ modeled on a finite or infinite dimensional cube $[0,1]^n$, $n\le \w$, we construct a meager $F_\sigma$-subset $X\subset M$ which is universal meager in the sense that for each meager subset $A\subset M$ there is a homeomorphism $h:M\to M$ such that $h(A)\subset X$. We also prove that any two universal meager $F_\sigma$-sets in $M$ are ambiently homeomorphic.
\end{abstract}
\maketitle

In this paper we shall construct and characterize universal meager $F_\sigma$-sets in $\II^n$-manifolds.

A meager subset $A$ of a topological space $X$ is called {\em universal meager} if for each meager subset $B\subset X$ there is a homeomorphism $h:X\to X$ such that $h(B)\subset A$. So, each universal meager subset of $X$ contains homeomorphic copies of all other meager subsets of $X$.

In fact, the notion of a universal meager set is a special case of a more general notion of a $\K$-universal set for some family $\K$ of subsets of a topological space $X$. Namely, we define a set $U\in\K$ to be {\em $\K$-universal} if for each set $K\in\K$ there is a homeomorphism $h:X\to X$ such that $h(K)\subset U$.

$\K$-Universal sets for various classes $\K$ often appear in topology. A classical example of such set is the Sierpi\'nski Carpet $M^2_1$, known to be a $\K$-universal set for the family $\K$ of all (closed) nowhere dense subsets of the square $\II^2=[0,1]^2$ (see \cite{S16}). The Sierpi\'nski Carpet $M^2_1$ is one of the Menger cubes $M^n_k$, which are $\K$-universal for the family $\K$ of all $k$-dimensional compact subsets of the $n$-dimensional cube $\II^n$ (see \cite{Stanko}, \cite[\S4.1]{Chi}). An analogue of the Sierpi\'nski Carpet exists also in the Hilbert cube $\II^\w$, which contains a $\Z_0$-universal set for the family $\Z_0$ of closed nowhere dense subsets of $\II^\w$ (see \cite{BR}).

Many $\K$-universal spaces arise in infinite-dimensional topology. For example, the pseudo-boundary $B(\II^\w)=[0,1]^\w\setminus (0,1)^\w$ of the Hilbert cube $\II^\w$ is known to be $\sigma\Z_\w$-universal for the family $\sigma\Z_\w$ of $\sigma Z_\w$-subsets of $\II^\w$. What is surprising, up to an ambient homeomorphism, $B(\II^\w)$ is a unique $\sigma\Z_\w$-universal set in $\II^\w$. In this paper we shall show that such a uniqueness theorem also holds for $\sigma\Z_0$-universal subsets in the Hilbert cube $\II^\w$.

Let us recall the definition of the families $\sigma\Z_\w$ and $\sigma\Z_0$. They consist of $\sigma Z_\w$-sets and $\sigma Z_0$-sets, respectively.

A closed subset $A$ of a topological space $X$ is called a {\em $Z_n$-set} in $X$ for a (finite or infinite) number $n\le \w$ if the set $\{f\in C(\II^n,X):f(\II^n)\cap A=\emptyset\}$ is dense in the space $C(\II^n,X)$ of all continuous functions $f:\II^n\to X$, endowed with the compact-open topology. Here by $\II=[0,1]$ we denote the unit interval and by $\II^n$ the $n$-dimensional cube. For $n=\w$ the space $\II^n=\II^\w$ is the Hilbert cube.

A subset $A\subset X$ is called a {\em $\sigma Z_n$-set} in $X$ if $A$ can be written as the  union $A=\bigcup_{k\in\w}A_k$ of countably many $Z_n$-sets $A_k\subset X$. Let us observe that a subset $A\subset X$ is a $Z_0$-set in $X$ if and only if it is closed and nowhere dense in $X$, and $A$ is a $\sigma Z_0$-set if and only if $A$ is a meager $F_\sigma$-set in $X$.

For a topological space $X$ by $\Z_n$ and $\sigma\Z_n$ we denote the families of $Z_n$-sets and $\sigma Z_n$-sets in  $X$, respectively.

A characterization of $\Z_\w$-universal sets in the Hilbert cube is quite simple and can be easily derived from the $Z$-Set Unknotting Theorem 11.1 from \cite{Chap}:

\begin{proposition} A subset $A\subset\II^\w$ is $\Z_\w$-universal in $\II^\w$ if and only if $A$ is a $Z_\w$-set in $\II^\w$, containing a topological copy of the Hilbert cube $\II^\w$.
\end{proposition}

A characterization of $\sigma\Z_\w$-universal sets in the Hilbert cube also is well-known and can be given in many different terms (skeletoid of Bessaga-Pelczynski \cite{BP}, capsets of Anderson \cite{An}, \cite{Chap71}, absorptive sets of West \cite{West}, pseudoboundaries of Geoghegan and Summerhill \cite{GS72}, \cite{GS74}). For our purposes the most appropriate approach is that of West \cite{West} and Geoghegan and Summerhill \cite{GS74}. To formulate this approach, we need to recall some notation.

Let $\U$, $\V$ be two families of sets of a topological space $X$.
Put
$$
\begin{aligned}
\U\wedge \V&=\{U\cap V:U\in\U,\;V\in\V,\;U\cap V\ne\emptyset\}\mbox{ and }\\
\U\vee \V&=\{U\cup V:U\in\U,\;V\in\V,\;U\cap V\ne\emptyset\}.
\end{aligned}
$$
We shall write $\U\prec\V$ and say that $\U$ {\em refines} $\V$ if each set $U\in\U$ is contained in some set $V\in\V$. Let $\St(\U,\V)=\{\St(U,\V):U\in\U\}$ where $\St(U,\V)=\bigcup\{V\in\V:U\cap V\ne\emptyset\}$. Put $\St(\U)=\St(\U,\U)$ and $\St^{n+1}(\U)=\St(\St^n(\U))$ for each $n>0$. We shall say that two maps $f,g:Z\to X$ are {\em $\U$-near} and denote it by $(f,g)\prec\U$ if the family $(f,g)=\big\{\{f(z),g(z)\}:z\in Z\big\}$ refines the family $\U\cup\big\{\{x\}:x\in X\big\}$.
For a family $\F$ of subsets of a metric space $(X,d)$ we put $\mesh(\F)=\sup_{F\in\F}\diam(F)$.

Let $\K$ be a family of closed subsets of a Polish space $X$ and $\sigma\K=\{\bigcup_{n\in\w}A_n:A_n\in\K,\;\;n\in\w\}$. We shall say that $\K$ is {\em topologically invariant} if $\K=\{h(K):K\in\K\}$ for each homeomorphism $h:X\to X$.

A subset $B\subset X$ is called {\em $\K$-absorptive} in $X$ if $B\in\sigma\K$ and for each set $K\subset\K$,  open set $V\subset X$, and  open cover $\U$ of $V$ there is a homeomorphism $h:V\to V$ such that $h(K\cap V)\subset B\cap V$ and $(h,\id)\prec\U$. An important observation is that each set $A\in\sigma\K$ containing a $\K$-absorptive subset of $X$ is also $\K$-absorptive.

The following powerfull uniqueness theorem was proved by West \cite{West} and Geoghegan and Summerhill  \cite[2.5]{GS74}.

\begin{theorem}[Uniqueness Theorem for $\K$-absorptive sets]\label{t1} Let $\K$ be a topologically invariant family of closed subsets of a Polish space $X$. Then any two $\K$-absorptive sets $B,B'\subset X$ are ambiently homeomorphic. More precisely, for any open set $V\subset X$ and any open cover $\U$ of $V$ there is a homeomorphism $h:V\to V$ such that $h(V\cap B)=V\cap B'$ and $h$ is $\U$-near to the identity map of $V$.
\end{theorem}

Two subsets $A,B$ of a topological space $X$ are called {\em ambiently homeomorphic} if there is a homeomorphism $h:X\to X$ such that $h(A)=B$. This happens if and only if the pairs $(X,A)$ and $(X,B)$ are homeomorphic. We shall say that two pairs $(X,A)$ and $(Y,B)$ of topological spaces $A\subset X$ and $B\subset Y$ are {\em homeomorphic} if there is a homeomorphism $h:X\to Y$ such that $h(A)=B$. In this case we say that $h:(X,A)\to (Y,B)$ is a homeomorphism of pairs.

According to the following corollary of Theorem~\ref{t1}, each $\K$-absorptive set is  $\sigma\K$-universal.

\begin{corollary}\label{c1} Let $\K$ be a topologically invariant family of closed subsets of a Polish space. If a $\K$-absorptive set $B$ in $X$ exists, then a subset $A\subset X$ is $\sigma\K$-universal in $X$ if and only if $A$ is $\K$-absorptive.
\end{corollary}

\begin{proof} Assume that a subset $A$ of $X$ is $\K$-absorptive. The definition implies that $A\in\sigma \K$. To show that $A$ is $\sigma\K$-universal, fix any subset $K\in\sigma\K$. The definition of a $\K$-absorptive set implies that the union $A\cup K$ is $\K$-absorptive. By the Uniqueness Theorem~\ref{t1}, there is a homeomorphism of pairs $h:(X,A\cup K)\to (X,A)$. This homeomorphism embeds the set $K$ into $A$, witnessing that the $\K$-absorptive set $A$ is $\sigma\K$-universal.

Now assume that a set $A\subset X$ is $\sigma \K$-universal. Since the $\K$-absorptive set $B$ belongs to the family $\sigma\K$, there is a homeomorphism $h$ of $X$ such that $h(B)\subset A$. The topological invariance of the class $\K$ implies that the set $h(B)$ is $\K$-absorptive, and so is the set $A\supset h(B)$.
\end{proof}

Corollary~\ref{c1} reduces the problem of studying $\sigma\K$-universal sets in a Polish space $X$ to studying $\K$-absorptive sets in $X$ (under the assumption that a $\K$-absorptive set in $X$ exists). The problem of the existence of $\K$-absorptive sets was considered in several papers. In particular, Geoghegan and Summerhill \cite{GS74}  proved that each Euclidean space $\IR^n$ contains a $\Z_0$-absorptive set and such a set is unique up to ambient homeomorphism.

Unfortunately, the methods of constructing $\Z_0$-absorptive sets in Euclidean spaces used in \cite{GS74} does not work in case of the Hilbert cube or Hilbert cube manifolds (in spite of the fact that the paper \cite{GS74} was written  to demonstrate applications of methods of infinite-dimensional topology in the theory of finite-dimensional manifolds). Known results on $\Z_\w$-absorptive sets in the Hilbert cube $\II^\w$ and $\Z_0$-absorptive sets in Euclidean spaces allow us to make the following:

\begin{conjecture}\label{conj1} The Hilbert cube contains a $\Z_n$-absorptive set for every $n\le\w$.
\end{conjecture}

This conjecture is true for $n=\w$ as witnessed by the pseudoboundary $B(\II^\w)=\II^\w\setminus (0,1)^\w$ of $\II^\w$ which is a $\Z_\w$-absorptive set in $\II^\w$. In this paper we shall confirm  Conjecture~\ref{conj1} for $n=0$. In fact, our proof works not only for the Hilbert cube but also for any $\II^k$-manifold of finite or infinite dimension. By a {\em manifold modeled on a space $E$} (briefly, an {\em $E$-manifold\/}) we understand any paracompact space $M$ admitting a cover by open subsets homeomorphic to open subspaces of the model space $E$. In this paper we consider only manifolds modeled on (finite or infinite dimensional) cubes $\II^n$, $n\le\w$. So, from now on, by a {\em manifold} we shall understand an $\II^n$-manifold for some $0<n\le\w$. If a manifold $X$ is finite-dimensional, then its {\em boundary} $\partial X$ consists of all points $x\in X$ which do not have neighborhoods homeomorphic to Euclidean spaces. If $X$ is a Hilbert cube manifold, then we put $\partial X=\emptyset$.

Our approach to constructing $\Z_0$-absorptive sets in manifolds is based on the notion of a tame $G_\delta$-set which is interesting by itself, see \cite{BMRZ}. First we recall some definitions.

A family $\F$ of subsets of a topological space $X$ is called {\em vanishing} if for each open cover $\U$ of $X$ the family $\F'=\{F\in\F:\forall U\in\U,\;\;F\not\subset U\}$ is locally finite in $X$. It is easy to see that a countable family $\F=\{F_n\}_{n\in\w}$ of subsets of a compact metric space $(X,d)$ is vanishing if and only if $\lim_{n\to\infty}\mathrm{diam}(F_n)=0$.

An open subset $B$ of an $\II^n$-manifold $X$ is called a {\em tame open ball} in $X$ if its closure $\bar B$  has on open neighborhood $O(\bar B)$ in $X$ such that the pair $(O(\bar B),\bar B)$ is homeomorphic to the pair $(\IR^n,\II^n)$ if $n<\w$ and to the pair $(\II^\w\times[0,\infty),\II^\w\times[0,1])$ if $n=\w$.
Tame balls form a neighborhood base at each point $x\in X$, which does not belong to the boundary $\partial X$ of $X$ (this is trivial for $n<\w$ and follows from Theorem~12.2 of \cite{Chap} for $n=\w$).

A subset $U$ of a manifold $X$ is called a {\em tame open set in $X$} if $U=\bigcup\U$ for some vanishing family $\U$ of tame open balls having pairwise disjoint closures in $X$. Observe that the family $\U$ is unique and coincides with the family $\C(U)$ of connected components of the set $U$. By $\bar \C(U)=\{\bar C:C\in\C(U)\}$ we shall denote the family of the closures of the connected components of $U$ in $X$.

A subset $G\subset X$ is called a {\em tame $G_\delta$-set in $X$} if $U=\bigcap_{n\in\w}U_n$ for some decreasing sequence $(U_n)_{n\in\w}$ of tame open sets such that the family $\C=\bigcup_{n\in\w}\C(U_n)$ is vanishing and for every $n\in\w$ the family  $\bar\C(U_{n+1})$ refines the family $\C(U_n)$ of connected components of $U_n$.

Tame open and tame $G_\delta$-sets  can be equivalently defined via tame families of tame open balls.
A family $\U$ of non-empty open subsets of a topological space $X$ is called {\em tame} if $\U$ is vanishing and for any distinct sets $U,V\in\U$ one of three possibilities hold: either $\bar U\cap\bar V=\emptyset$ or $\bar U\subset V$ or $\bar V\subset U$.
For a family $\U$ of subsets of a set $X$ by $$\ulim\U=\textstyle{\bigcap\big\{\bigcup}(\U\setminus\F):\F\mbox{ is a finite subfamily of $\U$}\big\}$$ we denote the set of all points $x\in X$ which belong to infinite number of sets $U\in\U$.

\begin{proposition}\label{p2} A subset $T$ of a manifold $X$ is tame open \textup{(}resp. tame $G_\delta$\textup{)} if and only if $T=\bigcup\mathcal T$ \textup{(}resp. $T=\ulim\mathcal T$~\textup{)} for a suitable tame family $\mathcal T$ of tame open balls in $X$.
\end{proposition}

\begin{proof} The ``only if'' part follows directly from the definition of a tame open (resp. tame $G_\delta$) set. To prove the ``if'' part, assume that $\mathcal T$ is a tame family of tame open balls in $X$. Endow the family $\mathcal T$ with a partial order $\le$ defined by the reverse inclusion relation, that is $U\le V$ if and only if $U\supset V$. The vanishing property of $\mathcal T$ guarantees that for each set $U\in\mathcal T$ the set ${\downarrow}U=\{V\in\mathcal T:V\le U\}$ is finite. This allows us to define the ordinal $\rank(U)$ letting $\rank(U)=|{\downarrow}{U}|$. For each number $n\in\w$ let $\mathcal T_n=\{U\in\mathcal T:\rank(U)=n+1\}$. It follows from the definition of a tame family that the union $U_n=\bigcup\mathcal T_n$ is a tame open set and $U_n\subset U_{n-1}$, where $U_{-1}=X$. In particular, the union $\bigcup\mathcal T=U_0$ is tame open set in $X$ and the set $T=\ulim\mathcal T=\bigcap_{n\in\w}U_n$ is a tame $G_\delta$-set in $X$.
\end{proof}

The classes of dense tame open sets and dense tame $G_\delta$-sets have the following cofinality property.

\begin{proposition}\label{p3}
\begin{enumerate}
\item Each open subset of a manifold $X$ contains a dense tame open set.
\item Each $G_\delta$-subset of a manifold contains a dense tame $G_\delta$-set.
\end{enumerate}
\end{proposition}

\begin{proof} Let $X$ be a manifold and $d$ be a metric generating the topology of $X$.
\smallskip

1. Given an open set $V\subset X$ and an open cover $\U$ of $V$ we shall construct
a tame open set $W\subset X$ such that $W$ is dense in $V$ and the family $\bar \C(W)$ refines the cover $\U$. Replacing $V$ by $V\setminus\partial X$, we can assume that the set $V$ does not intersect the boundary $\partial X$ of $X$. Replacing the set $V$ by $V\setminus\{v\}$ for some point $v\in V$, we can additionally assume that the set $V$ is not compact. We can also assume that $V=\bigcup \U$. Without loss of generality, the manifold $X$ is connected and hence separable. So, we can fix a countable dense subset $\{x_n\}_{n\in\w}$ in $V$.
By induction we can construct an increasing number sequence $(n_k)_{k\in\w}$ and a sequence $B_k$ of tame open balls in $X$ such that for each $k\in\w$ the following conditions hold:
\begin{enumerate}
\item $n_k$ is the smallest number $n$ such that $x_n\not\in \bigcup_{i<k}\bar B_k$;
\item $B_k$ is a tame open ball such that $x_{n_k}\in B_k$, the closure $\bar B_k$ of $B_k$ in $X$ has diameter $<2^{-k}$ and is contained in  $U\setminus \bigcup_{i<k}\bar B_k$ for some set  $U\in\U$.
\end{enumerate}
It is easy to check that $W=\bigcup_{k\in\w}B_k$ is a required dense tame open set in $V$ with $\bar\C(W)=\{\bar B_k\}_{k\in\w}\prec\U$.
\smallskip

2. Fix an arbitrary $G_\delta$-set $G$ in $X$ and write it as the intersection $G=\bigcap_{n\in\w}U_n$ of a decreasing sequence $(U_n)_{n\in\w}$ of open sets in $X$. By the (proof of the) preceding item, we can construct inductively a decreasing sequence $(V_n)_{n\in\w}$ of tame open sets in $X$ such that for every $n\in\w$ we get
\begin{itemize}
\item $\mesh\, \bar \C(V_n)<2^{-n}$,
\item $\bigcup\bar\C(V_n)\subset V_{n-1}\cap U_n$, and
\item $V_n$ is dense in $V_{n-1}\cap U_n$.
\end{itemize}
Here we assume that $V_{-1}=X$.
It follows that $\V=\bigcup_{n\in\w}\C(V_n)$ is a tame family of tame open balls whose limit set  $\ulim\V=\bigcap_{n\in\w}V_n$ is a required dense tame $G_\delta$-set in $G$.
\end{proof}

It is easy to see that any two tame open balls in a connected $\II^n$-manifold are ambiently homeomorphic. A similar fact also holds also for dense tame open sets.
Generalizing earlier results of Whyburn \cite{Whyburn} and Cannon \cite{Cannon}, Banakh and Repov\v s in \cite[Corollary 2.8]{BR} proved the following Uniqueness Theorem for dense tame open sets.

\begin{theorem}[Uniqueness Theorem for Dense Tame Open Sets in Manifolds]\label{t2} Any two dense tame open sets $U,U'\subset X$ of a manifold $X$ are ambiently homeomorphic. Moreover, for each open cover $\U$ of $X$ there is a homeomorphism $h:(X,U)\to(X,U')$ such that $(h,\id)\prec \St(\bar \C(U),\U)\vee \St(\bar \C(U'),\U)$.
\end{theorem}

This theorem will be our main tool in the proof of the following Uniqueness Theorem for dense tame $G_\delta$-sets.

\begin{theorem}[Uniqueness Theorem for Dense Tame $G_\delta$-Sets in Manifolds]\label{t3} Any two dense tame $G_\delta$-sets $G,G'$ in a manifold $X$ are ambiently homeomorphic. Moreover, for each open cover $\U$ of $X$ there is a homeomorphism $h:(X,G)\to(X,G')$ such that $(h,\id)\prec\U$.
\end{theorem}

\begin{proof} Fix a bounded complete metric $d$ generating the topology of the manifold $X$. By \cite[8.1.10]{En}, the metric $d$ can be chosen so that the cover $\{\bar B(x,1):x\in X\}$ by closed balls of radius 1 refines the cover $\U$. In this case any two functions $f,g:X\to X$ with $d(f,g)=\sup_{x\in X}d(f(x),g(x))\le 1$ are $\U$-near.

Represent the tame $G_\delta$-sets $G$ and $G'$ as the limit sets $G=\ulim \mathcal G$ and $G'=\ulim \mathcal G'$ of suitable tame families $\mathcal G$ and $\mathcal G'$ of tame open balls in $X$. For every $n\in\w$ let
$\mathcal G_n=\{U\in\mathcal G:|\{V\in\mathcal G:V\supset \bar U\}|\ge n\}$ and
$\mathcal G'_n=\{U\in\mathcal G':|\{V\in\mathcal G':V\supset \bar U\}|\ge n\}$. It follows that $G=\bigcap_{n\in\w}\bigcup\mathcal G_n$ and $G'=\bigcap_{n\in\w}\bigcup\mathcal G'_n$.

Let $U_{-1}=U_{-1}'=X$ and $h_{-1}:X\to X$ be the identity homeomorphism of $X$.
Let also $\U_{-1}=\U_{-1}'$ be a cover of $X$ by open subsets of diameter $\le\frac1{8}$.

For every $n\in\w$ we shall construct a
homeomorphism $h_n:X\to X$, two tame open sets $U_n,U_n'\subset X$, and open covers $\U_n$, $\U_n'$ of the sets $U_n$, $U_n'$, respectively, such that
\begin{enumerate}
\item $G\subset U_n\subset U_{n-1}\cap \bigcup\mathcal G_n$ and $\bar \C(U_n)\prec \U_{n-1}$;
\item $G'\subset U_n'\subset U_{n-1}'\cap\bigcup\mathcal G'_n$ and $\bar\C(U_n')\prec \U_{n-1}'\wedge h_{n-1}(\U_{n-1})$;
\item $h_n(U_n)=U_n'$;
\item $h_n|X\setminus U_{n-1}=h_{n-1}|X\setminus U_{n-1}$;
\item $d(h_n,h_{n-1})\le 2^{-n-1}$ and $d(h_n^{-1},h_{n-1}^{-1})\le2^{-n-1}$;
\item $\mesh(\U_n')<2^{-n-3}$, $\mesh(\U_n)<2^{-n-3}$, and $\St^2(\U_n)\prec\{B(x,d(x,X\setminus U_n)/2):x\in U_n\}$.
\end{enumerate}

Assume that for some $n\in\w$ the open sets $U_{n-1},U_{n-1}'$, open covers $\U_{n-1},\U_{n-1}'$  and a homeomorphism $h_{n-1}:(X,U_{n-1})\to (X,U_{n-1}')$ satisfying the conditions (1)--(6) have been constructed. Consider the subfamilies $\F_n=\{U\in\mathcal G_n:\{\bar U\}\prec\U_{n-1}\}$ and
$\F'_n=\{U\in\mathcal G_n':\{\bar U\}\prec \U_{n-1}'\wedge h_{n-1}(\U_{n-1})\}$. The vanishing property of the tame families $\mathcal G$ and $\mathcal G'$ implies that the sets $U_n=\bigcup\F_n$ and $U_n'=\bigcup\F_n'$ satisfy the conditions (1), (2) of the inductive construction. The sets $U_n$ and $U'_n$ are tame open, being unions of the tame families $\F_n$ and $\F_n'$, respectively. Moreover, $\bar\C(U_n)\prec \U_{n-1}$ and $\bar\C(U'_n)\prec \U_{n-1}'\wedge h_{n-1}(\U_{n-1})$.

Now we shall construct a homeomorphism $h_n:(X,U_n)\to (X,U_n')$. Since $h_{n-1}(U_{n-1})=U_{n-1}'$, each connected component $C\in\C(U_{n-1})$ of the open set $U_{n-1}$ maps onto the connected component $C'=h_{n-1}(C)\in\C(U'_{n-1})$ of the set $U_{n-1}'$. Taking into account that each set $\bar B\in\bar \C(U_n)$ is a compact connected subset of the open set $\bigcup\U'_{n-1}=U'_{n-1}$, we see that the intersection $U_n'\cap C'$ is a dense tame open set in the open set $C'$. Consequently, its image  $h_{n-1}^{-1}(U_n'\cap C')$ is a dense tame open set in the open set $C=h^{-1}_{n-1}(C')$. By Theorem~\ref{t2}, there is a homeomorphism of pairs $g_C:(C,C\cap U_n)\to (C,h_{n-1}^{-1}(C'\cap U_n'))$ which is $\W_C$-near to the identity map $\id_C:C\to C$ for the cover
 $\W_C=\St(\bar \C(C\cap U_n),\U_{n-1})\vee \St(\bar \C(h_{n-1}^{-1}(C'\cap U_n')),\U_{n-1})$.

Taking into account that $$\bar \C(C\cap U_n)\prec\bar\C(U_n)\prec \U_{n-1}\mbox { and }\bar \C(h_{n-1}^{-1}(U'_n\cap C'))\prec\bar\C(h_{n-1}^{-1}(U_n'))=h_{n-1}^{-1}(\bar \C(U'_n))\prec h_{n-1}^{-1}(h_{n-1}(\U_{n-1}))=\U_{n-1},$$ we conclude that
$$
\begin{aligned}
\W_C&=\St(\bar \C(C\cap U_n),\U_{n-1})\vee \St(\bar \C(h_{n-1}^{-1}(C'\cap U_n')),\U_{n-1})\prec\St(\U_{n-1},\U_{n-1})\vee\St(\U_{n-1},\U_{n-1})=\\
&=\St(\U_{n-1})\vee\St(\U_{n-1})\prec\St^2(\U_{n-1})\prec\{B(x,d(X\setminus U_{n-1})/2):x\in U_{n-1}\}.
\end{aligned}
$$

Now the vanishing property of the family $\C(U_{n-1})$ implies that the map $g_n:X\to X$ defined by $$g_n(x)=\begin{cases}
 x&\mbox{if $x\notin U_{n-1}$},\\
 g_C&\mbox{if $x\in C\in\C(U_{n-1})$}
 \end{cases}
 $$is a homeomorphism of $X$ such that $(g_n,\id)\prec\St^2(\U_{n-1})$ and $(g_n,\id)\prec\C(U_{n-1})$. Then $h_n=h_{n-1}\circ g_n$ is a homeomorphism of $X$ satisfying the conditions (3) and (4) of the inductive construction.

To prove the condition (5) we shall consider separately the cases of $n=0$ and $n>0$. If $n=0$, then $h_0=g_0$ and hence $(h_0,h_{-1})=(g_0,\id)\prec \St^2(\U_{-1})$. It follows from $\mesh(\U_{-1})\le 1/8$ that $d(h_0^{-1},h_{-1}^{-1})=d(h_0,h_{-1})\le \mesh(\St^2(\U_{-1}))\le \frac12$.

If $n>0$, then
$(h_n,h_{n-1})=(h_{n-1}\circ g_n,h_{n-1}\circ\id)\prec h_{n-1}(\C(U_{n-1}))=\C(U'_{n-1})\prec\U_{n-2}'
$
implies
$d(h_n,h_{n-1})\le \mesh(\U'_{n-2})\le 2^{-n-1}$.
By analogy,
$(h_{n}^{-1},h_{n-1}^{-1})=(g_n^{-1}\circ h^{-1}_{n-1},h^{-1}_{n-1})=(g_n^{-1},\id)=(g_n,\id)\prec\C(U_{n-1})\prec\U_{n-2}$ implies $d(h_n^{-1},h_{n-1}^{-1})\le \mesh(\U_{n-2})\le 2^{-n-1}$. So, the condition (5) holds.

Finally, using the paracompactness of the metrizable spaces $U_n$ and $U_n'$ choose two open covers $\U_n$ and $\U_n'$ of $U_n$ and $U_n'$ satisfying the condition (6).
\smallskip

After completing the inductive construction, we obtain a sequence of homeomorphisms $h_n:(X,U_n)\to (X,U_n')$, $n\in\w$. The condition (5) guarantees that the limit map $h=\lim_{n\to\infty}h_n$ is a well-defined homeomorphism of $X$ such that $d(h,\id)\le1$. Moreover, the conditions (1) and (3) imply
$$h(G)=h\big(\bigcap_{n\in\w}U_n\big)=\bigcap_{n\in\w}h(U_n)=\bigcap_{n\in\w}U_n'=G'.$$ By the choice of the metric $d$, the inequality $d(h,\id)\le 1$ implies $(h,\id)\prec\U$. So, $h:(X,G)\to(X,G')$ is a required homeomorphism of pairs with $(h,\id)\prec\U$.
\end{proof}

Now we are able to prove a characterization of $\sigma\Z_0$-universal sets in manifolds.

\begin{theorem}[Characterization of $\sigma\Z_0$-Universal Sets in Manifolds]\label{t4} For a subset $A$ of a manifold $X$ the following conditions are equivalent:
\begin{enumerate}
\item $A$ is $\sigma\Z_0$-universal in $X$;
\item $A$ is $\Z_0$-absorptive in $X$;
\item the complement $X\setminus A$ is a dense tame $G_\delta$-set in $X$.
\end{enumerate}
\end{theorem}

\begin{proof} We shall prove the equivalences $(3)\Leftrightarrow(2)\Leftrightarrow(1)$. Let $d$ be a metric generating the topology of the manifold $X$.
\smallskip

To prove that $(3)\Ra(2)$, assume that the complement $X\setminus A$ is a dense tame $G_\delta$-set in $X$. To prove that $A$ is $\Z_0$-absorptive, fix any open set $V\subset X$, an open cover $\U$ of $V$ and a closed nowhere dense subset $K\subset X$.
 We lose no generality assuming that $\U\prec\{B(x,d(x,X\setminus V)/2):x\in V\}$. Since $V\setminus (A\cup K)$ is a dense $G_\delta$-set in $V$, we can apply Proposition~\ref{p3} and find a dense tame $G_\delta$-set $G\subset V\setminus (A\cup K)$. The characterization of tame $G_\delta$-sets given in Proposition~\ref{p2} implies that the intersection $V\cap (X\setminus A)=V\setminus A$ is a dense tame $G_\delta$-set in $V$. By Theorem~\ref{t3}, there is a homeomorphism of pairs $h:(V,G)\to(V,V\setminus A)$ such that $(h,\id)\prec\U$. Since $\U\prec\{B(x,d(x,X\setminus V)/2):x\in V\}$, the homeomorphism $h$ of $V$ extends to a homeomorphism $\bar h:X\to X$ such that $\bar h|X\setminus V=\id$. Observing that $\bar h(V\cap K)\subset \bar h(V\setminus G)=V\cap A$, we see that the set $A$ is  $\Z_0$-absorptive.
\smallskip

To prove that $(2)\Ra(3)$, assume that the set $A$ is $\Z_0$-absorptive. By Proposition~\ref{p3}, the dense $G_\delta$-set $X\setminus A$ contains a dense tame $G_\delta$-set $G$ in $X$. Since $A\subset X\setminus G$, the set $X\setminus G\in\sigma\Z_0$ is $\Z_0$-absorptive. By the Uniqueness Theorem~\ref{t3}, there is a homeomorphism of pairs $h:(X,A)\to (X,X\setminus G)$. Then $X\setminus A=h(G)$ is a dense tame $G_\delta$-set in $X$, which completes the proof of the implication $(2)\Ra(3)$.
\smallskip

By Proposition~\ref{p3}, $X$ contains a dense tame $G_\delta$-set $G$ and by the implication $(3)\Ra(2)$ proved above the complement $X\setminus G$ is $\Z_0$-absorptive. Now Corollary~\ref{c1} yields the equivalence $(2)\Leftrightarrow(1)$.
\end{proof}

Theorem~\ref{t4} implies:

\begin{corollary} Each dense $G_\delta$-subset of a dense tame $G_\delta$-set in a manifold is tame.
\end{corollary}

We finish this paper by some open problems. It is clear that each tame $G_\delta$-set in a manifold is zero-dimensional. However,  not each zero-dimensional dense $G_\delta$-subset of the Hilbert cube $\II^\w$ is tame.

\begin{proposition} For any dense $G_\delta$-set $G\subset\II$ the countable product $G^\w$ is not a tame $G_\delta$-set in $\II^\w$.
\end{proposition}

\begin{proof} Assuming that $G^\w$ is tame, we can find a dense tame open set $T\subset\II^\w$ containing $G^\w$.  By Theorem~1.4 of \cite{BR}, the complement $S=\II^\w\setminus T$ is homeomorphic to the Hilbert cube and the boundary $\bar B\setminus B$ of each  tame open ball $B\in\C(T)$ in $\II^\w$ is a $Z_\w$-set in $S$. Let $\pr_n:\II^\w\to\II$, $n\in\w$, denote the projection of the Hilbert cube $\II^\w$ onto the $n$th coordinate. Since $\II^\w\setminus T\subset\bigcup_{n\in\w}\pr_n^{-1}(\II\setminus G)$, Baire Theorem yields a non-empty open subset $W\subset S$ such that $W\subset\pr_n^{-1}(\II\setminus G)$ for some $n\in\w$. Since $S$ is homeomorphic to the Hilbert cube, we can assume that the set $W$  is connected and hence is contained in $\pr_n^{-1}(t)$ for some point $t\in\II\setminus G$. Since the union $\Delta=\bigcup_{B\in\C(U)}\bar B\setminus B$ is a $\sigma\kern-1pt Z_\w$-set in $S$, we can chose a point $x_0\in W\setminus \Delta$. Choose an open neighborhood $U$ of $x_0$ in $\II^\w$ such that $U\cap S\subset W$ and $U\setminus \pr_n^{-1}(t)$ has at most two connected components.

Since the family $\C(T)$ is vanishing and $T=\bigcup\C(T)$ is dense in $\II^\w$, there are three pairwise distinct tame open balls $B_1,B_2,B_3\in\C(T)$ such that $\bar B_1\cup\bar B_2\cup \bar B_2\subset U$.
Since the set $U\setminus \pr_n^{-1}(t)$ has at most two connected components, there are two distinct indices $1\le i,j\le 3$ such that the balls $B_i$ and $B_j$ meet the same connected component $V$ of $U\setminus\pr_n^{-1}(t)$. Since $\bar B_i\setminus B_i\subset U\cap S\subset \pr_n^{-1}(t)$, the set $V\cap B_i$ is closed-and-open in the connected set $V$ and hence coincides with $V$. So, $V\subset B_i$.
By the same reason, $V\subset B_j$, which is not possible as the balls $B_i$ and $B_j$ are disjoint.
\end{proof}

\begin{problem} Can the countable power $G^\w$ of a dense $G_\delta$-set $G\subset\II$ be covered by countably many dense tame $G_\delta$-sets?
\end{problem}

By Smirnov's result \cite[5.2.B]{En}, the Hilbert cube $\II^\w$ can be covered by $\aleph_1$  zero-dimensional $G_\delta$-sets.

\begin{problem} What is the smallest cardinality of a cover of the Hilbert cube $\II^\w$ by  tame $G_\delta$-sets? Is it equal to $\aleph_1$? {\rm (By Theorem 1.6 of \cite{BMRZ} this cardinality does not exceed $\mathrm{add}(\mathcal M)$, the additivity of the ideal $\mathcal M$ of meager subsets on the real line.)}
\end{problem}


\begin{thebibliography}{}

\bibitem{An} R.D.~Anderson, {\em On sigma-compact subsets of infinite-dimensional manifolds}, unpublished manuscript.

\bibitem{BMRZ} T. Banakh, M. Morayne, R. Ra{\l}owski, S. {\.Z}eberski, {\em Topologically invariant $\sigma$-ideals on the Hilbert cube}, preprint\newline (http://arxiv.org/abs/1302.5658).

\bibitem{BR} T.~Banakh, D.~Repovs, {\em Universal nowhere dense subsets of locally compact manifolds}, preprint\newline (http://arxiv.org/abs/1302.5651).

\bibitem{BP} C.~Bessaga, A.~Pelczy\'nski, {\em Selected topics in infinite-dimensional topology}, PWN, Warsaw, 1975.

\bibitem{Cannon} J.W.~Cannon,  {\em A positional characterization of the (n-1)-dimensional Sierpinski curve in $S^n$} ($n\ne4$), Fund. Math. {\bf 79}:2 (1973), 107--112.

\bibitem{Chap71} T.A.~Chapman, {\em Dense sigma-compact subsets of infinite-dimensional manifolds}, Trans. Amer. Math. Soc. {\bf 154} (1971) 399--426.

\bibitem{Chap}  T.A.~Chapman, {\em Lectures on Hilbert cube manifolds}, Amer. Math. Soc., Providence, R. I., 1976.

\bibitem{Chi} A.~Chigogidze, {\em Inverse spectra}, North-Holland Publishing Co., Amsterdam, 1996.

\bibitem{En} R.~Engelking, {\em General Topology}, Heldermann Verlag, Berlin, 1989.

\bibitem{En2} R.~Engelking, {\em Theory of dimensions finite and infinite}, Heldermann Verlag, Lemgo, 1995.

\bibitem{GS72} R.~Geoghegan, R.~Summerhill, {\em Infinite-dimensional methods in finite-dimensional geometric topology}, Bull. Amer. Math. Soc. {\bf 78} (1972), 1009--1014.

\bibitem{GS74} R.~Geoghegan, R.~Summerhill, {\em Pseudo-boundaries and pseudo-interiors in Euclidean spaces and topological manifolds}, Trans. Amer. Math. Soc. {\bf 194} (1974), 141--165.



\bibitem{Menger} K.~Menger, {\em Allgemeine Raume und Cartesische Raume Zweite Mitteilung: "Uber umfassendste $n$-dimensional Mengen"}, Proc. Akad. Amsterdam, {\bf 29} (1926), 1125--1128.

\bibitem{S16} W.~Sierpi\'nski, {\em Sur une courbe cantorienne qui contient une image biunivoque et continue de toute courbe donn\'ee}, C.R. Acad. Sci., Paris, {\bf 162} (1916), 629--632.

\bibitem{Stanko} M.A.~\v Stanko, {\em The embedding of compact into Euclidean space}, Math USSR Sbornik {\bf 12} (1970), 234--254.

\bibitem{West} J.~West, {\em The ambient homeomorphy of an incomplete subspace of infinite-dimensional Hilbert spaces}, Pacific J. Math. {\bf 34} (1970) 257--267.

\bibitem{Whyburn} G.~Whyburn, {\em Topological characterization of the Sierpinski curve}, Fund. Math. {\bf 45} (1958), 320--324.

\end{thebibliography}
\end{document}